\documentclass[twoside,11pt]{article}

\usepackage[preprint]{jmlr2e}
\usepackage[english]{babel}
\usepackage[utf8]{inputenc}
\usepackage[T1]{fontenc}

\usepackage{amsmath, enumitem, xcolor}

\usepackage{pgfplots}
\usepackage{tikz}
\allowdisplaybreaks

\newcommand{\calF}{\mathcal{F}}

\newcommand{\calX}{\mathcal{X}}

\newcommand{\bbH}{\mathbb{H}}

\newcommand{\bbR}{\mathbb{R}}

\newcommand{\KL}{D_{\mathrm{KL}}}

\newcommand{\eps}{\epsilon}

\newcommand{\bs}[1]{\boldsymbol{#1}}

\DeclareMathOperator{\cov}{cov}

\DeclareMathOperator{\E}{E}
\let\P\undefined
\DeclareMathOperator{\P}{P}

\DeclareMathOperator{\tr}{tr}

\newcommand{\ff}{{\bs{f\!f}}}

\newcommand{\ip}[1]{\langle #1 \rangle}
\newcommand{\qv}[1]{\left\langle #1 \right\rangle}

\newcommand{\citeb}[1]{\citeauthor{#1}, \citeyear{#1}}

\jmlrheading{23}{2022}{1-26}{09/2021; Revised 03/2022}{??}{??}{Dennis Nieman, Botond Szabo and Harry van Zanten}

\ShortHeadings{Contraction rates for variational GP regression}{Nieman, Szabo and van Zanten}
\firstpageno{1}

\begin{document}

\title{Contraction rates for sparse variational approximations in Gaussian process
regression}

\author{\name Dennis Nieman\footnote{corresponding author} \email d.nieman@vu.nl \\
       \addr Department of Mathematics\\
       Vrije Universiteit Amsterdam\\
De Boelelaan 1111, 1081 HV Amsterdam\\
The Netherlands
       \AND
       \name Botond Szabo \email botond.szabo@unibocconi.it\\
       \addr Department of Decision Sciences,\\
Bocconi Institute for Data Science and Analytics,\\
       Bocconi University\\
       Via Roentgen 1, Milano, Italy
\AND 
\name Harry van Zanten \email j.h.van.zanten@vu.nl \\
       \addr Department of Mathematics\\
       Vrije Universiteit Amsterdam\\
De Boelelaan 1111, 1081 HV Amsterdam\\
The Netherlands}

\editor{Marc Peter Deisenroth}

\maketitle

\begin{abstract}
We study the theoretical properties of a variational Bayes method in the Gaussian Process regression model. We consider the inducing variables method introduced by \cite{titsias2009a} and derive sufficient conditions for obtaining contraction rates for the corresponding variational Bayes (VB) posterior. As examples we show that for three particular covariance kernels (Matérn, squared exponential, random series prior) the VB approach can achieve optimal, minimax contraction rates for a sufficiently large number of appropriately chosen inducing variables. The theoretical findings are demonstrated by numerical experiments.\\
\end{abstract}


\let\thefootnote\relax\footnotetext{\makebox[0cm][l]{\makebox[-.2cm][r]{$*$}}Corresponding author.}

\section{Introduction} \label{sec: intro}
Suppose we observe $n$ independent pairs $(x_1, y_1), \ldots, (x_n, y_n)$, where each $x_i$ 
has distribution $G$ on a subset $\calX \subseteq \bbR^d$ and 
\begin{equation}\label{eq: setting}
y_i = f(x_i) + \varepsilon_i,\qquad i =1, \ldots, n,
\end{equation}
with an unknown function $f: \calX \to \bbR$ and $\varepsilon_1, \ldots, \varepsilon_n$ independent Gaussian variables with mean zero and variance $\sigma^2$. In Gaussian Process (GP) regression we model $f$ a-priori as a centered GP with covariance function $k : \calX \times \calX \to \bbR$.
GP regression has become popular due to the explicit expressions for the posterior (also a GP, see e.g.\ \citeb{Rasmussen2006} and Section \ref{sec: variational} ahead) and the marginal likelihood, and the ease with which uncertainty quantification can be obtained. Moreover, there exist mathematical guarantees for consistency, optimal contraction rates, and validity of uncertainty quantification (e.g. \citeb{vaart2008}; \citeb{sniekers2015}; \citeb{rousseau2017}).

A drawback of plain GP regression is the fact that computation of the posterior requires inversion of an $n\times n$ matrix, which becomes computationally demanding for large sample size $n$. The computational cost typically scales as $n^3$, which can be prohibitive in practice. To alleviate the computational burden, reduced rank approximations are often employed; see for instance Chapter 8 of \cite{Rasmussen2006} and the more recent overview in \cite{liu2020}. These approximations somehow summarise the posterior using $m \ll n$ variables instead of $n$, typically reducing the order of the computational cost from $n^3$ to $nm^2$. 

In this paper we consider the {variational approximation} proposed by \cite{titsias2009a}. This approach uses $m$ so-called {inducing variables} to summarise the posterior (details are given in the next section). It is a true variational Bayes procedure, in the sense that the approximate posterior minimises the Kullback-Leibler (KL) divergence between the true posterior and a parametrised family of approximating distributions. 

While the computational aspects of low rank approximations are well understood, little is known about whether the mathematical guarantees for the true posterior carry over to the approximate posterior. \cite{burt2019} analyse the expected KL-divergence between the posterior and its variational approximation. In particular, they investigate in various cases how large the number of inducing variables $m$ should be chosen in relation to the sample size $n$ in order to ensure that the expected KL-divergence vanishes as $n$ becomes large. However, since the expectation that is considered is computed both over the data $(\bs{x},\bs{y})$ and over the prior on $f$, these results do not translate to (frequentist) guarantees about consistency and contraction rates, which assume that the data is generated from a fixed, ``true'' regression function $f_0$.

In this paper we derive contraction rates for the approximate posterior in this frequentist setup. This makes it possible to compare rates with known minimax lower bounds, which 
explain what the best possible contraction rates are and how these depend on global characteristics of the true regression function $f_0$, like its degree of smoothness. This in turn gives insight into how $m$ should be chosen in order for the variational posterior to have the same contraction rate as the true posterior.

Our findings can be summarised as follows:
\begin{enumerate}[label=(\roman*)]
\item
In order to have an optimal rate of contraction of the variational posterior 
around the true regression function $f_0$, it is not necessary that 
the KL-divergence between the true posterior and the variational approximation 
vanishes as $n \to \infty$.  
\item
For appropriately chosen inducing variables, one can recover an $\alpha$-smooth regression function $f_0$ at the optimal rate with the VB method using the Mat\'ern kernel or a series kernel with regularity hyper-parameter $\alpha$ if the number of inducing variables $m$ scales at least as $n^{d/(d+2\alpha)}$.
\item
These inducing variable VB methods also result in minimax contraction rate around $\alpha$-smooth regression functions $f_0$ for GP priors with squared exponential covariance kernel (in $d=1$) if the number of inducing variables $m$ scales at least as $n^{1/(1+2\alpha)}\log n$.
\item
Choosing fewer inducing points than the optimal number can result in overly smooth posterior means and conservative, sub-optimally large credible sets; see the numerical study in Section \ref{sec:num}.
\end{enumerate}

The remainder of the paper is organised as follows. In Section \ref{sec: variational} we recall the inducing variable variational Bayes method by \cite{titsias2009a}. Next in Section \ref{sec: GP:contraction} we briefly discuss contraction rate results for GP posteriors following \cite{vaart2008}. A more detailed description of the frequentist analysis of general (nonparametric) posteriors are given in Appendix A. The main results are presented in Section \ref{sec:main} where sufficient conditions are given on the GP and the inducing variables to obtain the contraction rate of the corresponding VB posterior.  In Sections \ref{sec:empeigen} and \ref{sec:opeigen} two specific choices of the inducing variables are described, from the eigendecompositions of respectively the covariance matrix and the covariance operator. We show in Section \ref{sec:examples} that these approaches result in rate optimal VB posterior contraction rates for the squared exponential, Mat\'ern and series covariance kernels, matching the optimal behaviour of the appropriately scaled  true posterior. Finally we conclude our results with a brief numerical study in Section \ref{sec:num}.

\subsection{Notation} For two positive sequences $a_n,b_n$ we use the notation $a_n\lesssim b_n$ if there exists a positive constant $C$ such that $a_n\leq C b_n$ for all $n$. We write $a_n\asymp b_n$ if $a_n\lesssim b_n$ and $b_n\lesssim a_n$ are satisfied simultaneously. We denote by $\tr$ the trace operator and by $\KL(\mu,\nu)$ the Kullback-Leibler divergence between the measures $\mu$ and $\nu$. The norm $\|\cdot\|$ denotes the Euclidean norm for vectors and the spectral/operator norm for matrices. By $L^2(\calX,G)$ we denote the space of (almost sure equivalence classes of) Borel measurable real-valued functions $f$ on $\calX$ such that $\|f\|^2_{2,G} := \int_\calX |f|^2 \,dG$ is finite.

\section{Inducing variables variational Bayes} \label{sec: variational}

In this section we recall the sparse GP regression approach of \cite{titsias2009a}, introducing the notation that we use throughout the paper. 

In the regression model \eqref{eq: setting}, if a centered GP prior $\Pi$ with covariance kernel $k$ is used, then the true posterior 
is again a GP, with mean and covariance function  given by 
\begin{gather*}
x \mapsto K_{x\bs{f}}(\sigma^2I + K_\ff)^{-1} \bs{y}, \\
(x,y) \mapsto k(x,y) - K_{x\bs{f}}(\sigma^2I + K_\ff)^{-1}K_{\bs{f}y},
\end{gather*}
respectively. Here we denote $\bs{y} = (y_1, \ldots, y_n)$, $\bs{f}=(f(x_1),\ldots,f(x_n))$, 
\begin{gather}
K_{x\bs{f}} = \cov_\Pi(f(x),\bs{f}) = (k(x, x_1), \ldots, k(x, x_n)) = K_{\bs{f}x}^T, \nonumber \\
K_\ff = \cov_\Pi(\bs{f},\bs{f}) = [k(x_i, x_j)]_{1\leq i,j \leq n}, \label{e:kff}
\end{gather}
where we emphasise through the subscript $\Pi$ that the covariances are computed under the prior $\Pi$ (and not (also) the distribution $G$ of the design points). We denote the posterior probability kernel by $\Pi(\,\cdot \mid \bs{x},\bs{y})$.

The idea of \cite{titsias2009a} is to summarise the true posterior through a collection of {inducing variables} $u_1,\ldots,u_m \in L^2(\Pi)$, which 
by definition are continuous linear functionals of the prior process on $f$. By the linearity assumption, the prior process $f$ conditional on $\bs{u} = (u_1,\ldots,u_m)$ is again a GP, with  mean and covariance function
given by 
\begin{gather}
x \mapsto K_{x\bs{u}}K_{\bs{uu}}^{-1}\bs{u}, \label{eq: fgu1}\\
(x,y) \mapsto k(x,y) - K_{x\bs{u}}K_{\bs{uu}}^{-1}K_{\bs{u}y}, \label{eq: fgu2}
\end{gather}
where $K_{x\bs{u}} = \cov_\Pi(f(x),\bs{u}) = K_{\bs{u}x}^T$ and $K_{\bs{uu}} = [\cov_\Pi(u_i,u_j)]_{1 \leq i,j \leq m}$. 
This motivates the construction of 
a variational family of measures approximating the posterior by postulating that the vector $\bs{u}$ has a Gaussian distribution with some mean $\mu \in \bbR^m$ and $m\times m$ covariance matrix $\Sigma$, and that the conditional $f\,|\, \bs{u}$ is the GP law given by 
\eqref{eq: fgu1}-\eqref{eq: fgu2}. This results in a variational family of GP laws indexed by variational parameters $\mu$ and $\Sigma$. Explicitly, for fixed $\mu$ and $\Sigma$, the variational approximation to the posterior is a GP with mean and covariance function given by 
\begin{gather*}
x \mapsto K_{x\bs{u}}K_{\bs{uu}}^{-1}{\mu},\\
(x,y) \mapsto k(x,y)-K_{x\bs{u}}K_{\bs{uu}}^{-1}(K_{\bs{uu}}-\Sigma)K_{\bs{uu}}^{-1}K_{\bs{u}y},
\end{gather*}
cf. also equation (2) in \cite{burt2019}. We denote this member of the variational family by $\Psi_{\mu, \Sigma}(\,\cdot \mid \bs{x},\bs{y})$.

It can be shown that for all $\mu$ and $\Sigma$, the approximation $\Psi_{\mu, \Sigma}(\,\cdot \mid \bs{x},\bs{y})$ and the true posterior $\Pi(\,\cdot \mid \bs{x},\bs{y})$ are equivalent measures (the Radon-Nikodym derivative reduces to a finite-dimensional Gaussian derivative, a function of only $m+n$ variables). Hence their Kullback-Leibler divergence is well defined. \cite{titsias2009} proves that there exist optimal $\mu'$ and $\Sigma'$ such that
\begin{multline}\label{eq: titsias}
\inf_{\mu, \Sigma} \KL \Big(\Psi_{\mu, \Sigma}(\,\cdot \mid \bs{x},\bs{y}) \,\Big\|\, \Pi(\,\cdot\mid \bs{x},\bs{y})\Big) = 
\KL \Big(\Psi_{\mu', \Sigma'}(\,\cdot\mid \bs{x},\bs{y}) \,\Big\|\, \Pi(\,\cdot\mid \bs{x},\bs{y})\Big)\\
 = \frac12 \Big(\bs{y}^T(Q_n^{-1}- K_n^{-1})\bs{y} + \log\frac{|Q_n|}{|K_n|}
+\frac1{\sigma^2}\tr(K_n-Q_n)\Big). 
\end{multline}
Here $K_n = \sigma^2I + K_\ff$ and $Q_n = \sigma^2I + Q_\ff$, where
\begin{equation}\label{e:qff}
Q_\ff = K_{\bs{fu}}K^{-1}_{\bs{uu}}K_{\bs{uf}}
\end{equation}
with $K_{\bs{uf}} = \cov_\Pi(\bs{u},\bs{f})$. Even though in \cite{titsias2009} the considered distributions are jointly over $f$ and $\bs u$, the Kullback-Leibler divergence does not change when we use the $f$-marginal distributions, as follows from \cite{matthews2016}, noting that the inducing variables are measurable functions of $f$.

The variational posterior 
$\Psi_{\mu', \Sigma'}(\,\cdot\mid \bs{x},\bs{y})$ can be seen as a particular rank-$m$ approximation 
of the full posterior $\Pi(\,\cdot\mid \bs{x},\bs{y})$. 
In the next section we present results  about the rate at which it contracts around the true regression function $f_0$
as $n \to \infty$. 
Since the precise form of the optimal variational parameters is not important here, we simply denote the 
variational posterior by $\Psi(\,\cdot\mid \bs{x},\bs{y}) = \Psi_{\mu', \Sigma'}(\,\cdot\mid \bs{x},\bs{y})$.

\newpage\section{Posterior contraction rates for Gaussian process priors}\label{sec: GP:contraction}
We give a brief overview of posterior contraction rates for GP priors. In Appendix \ref{a:gpcon} we provide further details and discuss general contraction rate results for (nonparametric) Bayesian methods. Here we focus on the results directly used in our main theorem in the upcoming section.

We study the posterior distribution $\Pi(\,\cdot \mid \bs x,\bs y)$ under the assumption that the data $(\bs x, \bs y)$ are generated according to some fixed, ``true'' regression function $f_0 \in L^2(\calX,G)$. In other words, we suppose \eqref{eq: setting} holds with $f_0$ instead of $f$, or equivalently, the pairs $(x_i,y_i)$ are i.i.d. with density 
\[ p_{f_0}(x,y) = (2\pi\sigma^2)^{-1/2}\exp(-(y-f_0(x))^2/(2\sigma^2)) \]
relative to the product of the probability measure $G$ and the Lebesgue measure. We denote by $\P_0$ the associated joint distribution of the data and by $\E_0$ its according expectation operator. General theory on Bayesian contraction rates gives conditions under which the posterior corresponding to a GP prior in the nonparametric regression model contracts around the true regression function $f_0$ at a certain rate $\eps_n \to 0$ as the sample size $n$ tends to infinity.

The standard approach for establishing contraction rates, as exposed in \cite{ghosal2017}, relies on the existence of appropriate hypothesis tests. This is guaranteed when the chosen metric is the Hellinger distance, so the contraction rate is naturally measured relative to this metric on the space of joint densities of the pair $(x_i, y_i)$. Given $f_1,f_2 \in L^2(\calX,G)$, this Hellinger distance $d_{\mathrm H}$ between the two associated densities $p_{f_1}, p_{f_2}$ is given by
\begin{align}
d_{\mathrm H}(p_{f_1}, p_{f_2})^2 
& =  \frac 12 \iint \Big(\sqrt{p_{f_1}(x,y)} - \sqrt{p_{f_2}(x,y)}\Big)^2 \,dy\,dG(x) \nonumber \\
& = \int_\calX 1- \exp\Big(-\frac{(f_1(x)-f_2(x))^2}{8\sigma^2}\Big) \,dG(x). \label{e:hel}
\end{align}
Considering this as a function of $(f_1,f_2)$, the distance $d_{\mathrm H}$ can be viewed as a metric on the function space $L^2(\calX,G)$. In the sequel we shall abuse our notation and simply write $d_{\mathrm H}(f_1,f_2)$.

The posterior is said to contract around the truth $f_0$ at the rate $\eps_n$ with respect to the Hellinger distance $d_{\mathrm H}$  if for all sequences $M_n \to \infty$, 
\begin{equation}\label{eq: postrate}
\E_0 \Pi\big(f: d_{\mathrm H}(f, f_0) \ge M_n\eps_n \mid \bs{x},\bs{y}) \to 0
\end{equation}
as $n \to \infty$. Loosely speaking, (\ref{eq: postrate}) entails that if $f_0$ generated the data, then, asymptotically, all posterior mass lies in Hellinger balls around $f_0$ with a radius of the order $\eps_n$. 

In view of \cite{vaart2008}, for GPs the posterior contraction rate is determined by the concentration function $\varphi_{f_0} : (0,\infty) \to \bbR$ associated to the GP prior $\Pi$, which is defined as
\begin{equation}\label{eq: conf}
\varphi_{f_0}(\eps) = \inf_{h\in \bbH:\|h-f_0\|_{2, G} \le \eps} \|h\|^2_\bbH -\log \Pi(f: \|f\|_{2, G} \le \eps). 
\end{equation}
Here $\bbH$ is the Reproducing Kernel Hilbert Space (RKHS) associated to the prior, and $\|\cdot\|_\bbH$ is the corresponding RKHS norm (see e.g. \citeb{rkhs} or Appendix I of \citeb{ghosal2017}). Specifically, if $\eps_n \to 0$ is such that $n\eps^2_n \to \infty$ and 
\begin{equation}\label{eq: con}
\varphi_{f_0}(\eps_n) \le n\eps^2_n, 
\end{equation}
then the posterior distribution contracts at the rate $\eps_n$. The following is a slightly more refined version of this statement.

\begin{lemma}\label{l:con}
Suppose that the concentration function inequality \eqref{eq: con} holds for some sequence of positive numbers $\epsilon_n \to 0$ with $n\epsilon_n^2 \to \infty$. Then, for every constant $C_2>0$ there exists an event $A_n$ in the $\sigma$-field generated by $(\bs x, \bs y)$ such that $\P_0(A_n) \to 1$ and
\begin{equation}\label{e:conbd}
\E_0 \Pi(f : d_{\mathrm H}(f,f_0) \geq M_n\epsilon_n \mid \bs x,\bs y)1_{A_n} \lesssim \exp(-C_2 n\epsilon_n^2).
\end{equation}
\end{lemma}

Note that the preceding lemma implies the posterior contraction \eqref{eq: postrate}. This inequality together with a bound on the Kullback-Leibler divergence in \eqref{eq: titsias} will help establish our main result, a contraction rate statement for the variational posterior; see Theorem \ref{thm: main} ahead and the lemmas below it.  The results leading to Lemma \ref{l:con} are recalled and discussed in Appendix \ref{a:gpcon}. We note that in specific examples, verifying the concentration inequality \eqref{eq: con} means analysing the so-called small ball behaviour of the prior GP and the approximation properties of its RKHS (see also Section \ref{sec:examples} and Appendix \ref{a:sqexp}).

\section{Main results}\label{sec:main}

In this paper we are interested in contraction rate results like \eqref{eq: postrate}, but for the variational posterior $\Psi(\,\cdot \mid \bs{x},\bs{y})$ instead of the full posterior $\Pi(\,\cdot\mid \bs{x},\bs{y})$. It is intuitively clear that in addition to an assumption like \eqref{eq: con}, this requires control over the approximation properties of the variational family, which depend on the choice of inducing variables $\bs{u} = (u_1,\ldots,u_n)$. In the following theorem, this is measured in terms of the expected ``size'' of the difference between the matrices $K_\ff$ and $Q_\ff$ (defined in \eqref{e:kff} and \eqref{e:qff}, respectively), which is the covariance matrix of the conditional law of the vector $\bs{f} = (f(x_1),\ldots, f(x_n))$ given $\bs{u}$ (see \eqref{eq: fgu2}). The size of $K_\ff- Q_\ff$ measures how well the vector of inducing variables $\bs{u}$ summarises the full prior distribution. In short, we characterise the contraction rate of the variational posterior by conditions on the inducing variables and the prior.

Below, $\|A\|$ and $\tr(A)$ are the operator norm and trace of the square matrix $A$, and $\E_{\bs{x}}$ is the expectation over the input variables $\bs{x}$ alone. 

\begin{theorem}\label{thm: main}
Suppose that for $f_0 \in L^2(\calX,G)$ and $\eps_n \to 0$ such that $n\eps^2_n \to \infty$, the concentration function inequality \eqref{eq: con} holds. If in addition 
there exists a constant $C > 0$ (independent of $n$) such that
\begin{align}
\E_{\bs{x}} \|K_\ff- Q_\ff\| & \le C, \label{eq:norm}\\
\E_{\bs{x}} \tr(K_\ff- Q_\ff) & \le Cn\eps^2_n, \label{eq:trace}
\end{align}
then the variational posterior contracts around $f_0$ at the rate $\eps_n$, that is,
for all sequences $M_n \to \infty$, 
\begin{equation}\label{e:vpcon}
\E_0 \Psi\big(f: d_{\mathrm H}(f, f_0) \ge M_n\eps_n \mid \bs{x},\bs{y}\big) \to 0.
\end{equation}
as $n \to \infty$. 
\end{theorem}

\begin{proof}
The concentration inequality also holds with $M_n\eps_n$ instead of $\eps_n$. Hence, by Lemma \ref{l:con}, there exist events $A_n$ and a constant $C_2 > 0$ such that $\P_0(A_n) \to 1$ and 
\[
\E_0 \Pi\big(f: d_{\mathrm H}(f, f_0) \ge M_n\eps_n \mid \bs{x},\bs{y} \big) 1_{A_n} \lesssim e^{-C_2nM_n^2\eps^2_n}. 
\]
Lemma \ref{l:rs} applied with $\delta_n = C_2 n M_n^2\epsilon_n^2$ yields
\[ \E_0 \Psi(f : d_{\mathrm H}(f,f_0) \geq M_n\eps_n \mid \bs x,\bs y) 1_{A_n} \lesssim \frac{\E_0 \KL (\Psi(\,\cdot \mid \bs{x}, \bs{y}) \,\|\, \Pi(\,\cdot\mid\bs{x},\bs{y})) + e^{-C_2nM_n^2\eps_n^2}}{nM_n^2\eps_n^2}. \]
The proof is completed by combining this with $\P_0(A_n^c) \to 0$, and, as we prove now,
\begin{equation}\label{e:klbd}
\E_0  \KL (\Psi(\,\cdot \mid \bs{x}, \bs{y}) \,\|\, \Pi(\,\cdot\mid\bs{x},\bs{y})) 
\le C_1n\eps^2_n
\end{equation}
for some positive constant $C_1$. By the concentration function inequality, there exist an $h \in \bbH$ such that $\|h\|^2_\bbH \le n\eps^2_n$ and $\|f_0-h\|_{2, G} \le \eps_n$. Applying Lemma \ref{lem: kl} ahead with that 
choice for $h$ and using the assumptions on $K_\ff- Q_\ff$ then establishes \eqref{e:klbd}.
\end{proof}

It can be seen from the proof that the variational posterior contracts at the same rate as the true posterior if the inequality \eqref{e:klbd} holds. Since $n\epsilon_n^2 \to \infty$, this means that the Kullback-Leibler divergence need not go to zero in $\P_0$-expectation. The inequality, which is an essential step in the above proof, follows from the next lemma. A crucial difference 
with Lemma 2 of \cite{burt2019} is that we consider $f_0$ to be fixed. 

\begin{lemma}\label{lem: kl}
For every $f_0 \in L^2(\mathcal X,G)$ and $h \in \bbH$ we have 
\begin{multline*}
\E_0 \KL \Big(\Psi(\,\cdot \mid \bs{x},\bs{y}) \,\Big\|\, \Pi(\,\cdot \mid \bs{x},\bs{y})\Big) \\
\le  
\frac1{\sigma^2} \Big( n\|f_0-h\|^2_{2, G} + \|h\|^2_\bbH \E_{\bs{x}}\|K_\ff- Q_\ff\|  
 + \E_{\bs{x}} \tr(K_\ff- Q_\ff)\Big).
\end{multline*}
\end{lemma}
\begin{proof}
The matrix $K_n-Q_n = K_\ff- Q_\ff$ is the covariance matrix of the conditional law of the vector
$\bs{f} = (f(x_1),\ldots, f(x_n))$ given $\bs{u} = (u_1,\ldots,u_m)$. In particular it is positive semidefinite, which implies that $K_n \ge Q_n$, hence $\log (|Q_n|/|K_n|)\leq 0$. Therefore, the KL-divergence between the variational class and the true posterior can be bounded from above by leaving out the logarithmic term on the right hand side of the identity \eqref{eq: titsias}, i.e.
\begin{multline}\label{eq: klbound}
\KL \Big(\Psi_{\mu', \Sigma'}(\,\cdot\mid \bs{x},\bs{y}) \,\Big\|\, \Pi(\,\cdot\mid \bs{x},\bs{y})\Big)
 \leq \frac12 \Big(\bs{y}^T(Q_n^{-1}- K_n^{-1})\bs{y} +\frac1{\sigma^2}\tr(K_n-Q_n)\Big). 
\end{multline}

Now let $\E_{\bs{y}}$ be the expectation over $\bs{y}$, assuming the input variables $\bs{x}$ 
are fixed and $f_0$ is the true regression function, so that $\E_0 = \E_{\bs{x}}\E_{\bs{y}}$. 
We have 
\begin{equation}\label{eq: piet}
\E_{\bs y} {\bs y}^T(Q_n^{-1} - K_n^{-1}){\bs y} = 
\bs{f}_0^T(Q_n^{-1}-K_n^{-1})\bs{f}_0 + \sigma^2 \tr(Q_n^{-1}-K_n^{-1}).
\end{equation}
For the first term on the right-hand side we write, with $\bs{h} = (h(x_1),\ldots, h(x_n))$,  
\begin{align*}
\frac12 \bs{f}_0^T (Q_n^{-1}-K_n^{-1}) \bs{f}_0
&\leq \bs{h}^T Q_n^{-1}(K_n-Q_n)K_n^{-1} \bs{h} + (\bs{f}_0-\bs{h})^T(Q_n^{-1}-K_n^{-1})(\bs{f}_0-\bs{h}) \\
&\leq \|Q_n^{-1}\|\|K_n-Q_n\| \bs{h}^T K_n^{-1} \bs{h} + (\bs{f}_0-\bs{h})^TQ_n^{-1}(\bs{f}_0-\bs{h}) \\
&\leq \frac{1}{\sigma^2}\Big( \|K_n-Q_n\| \bs{h}^T K_\ff^{-1} \bs{h} + \sum_{i=1}^n(f_0(x_i)-h(x_i))^2 \Big),
\end{align*}
where we used that $K_n = \sigma^2 I + K_\ff \geq K_\ff$ and $Q_n = \sigma^2I + Q_\ff \geq \sigma^2 I$. The quantity $\bs{h}^T K_\ff^{-1}\bs{h}$ is the squared RKHS norm of the orthogonal projection in $\bbH$ of the function $h$ on the linear span of the functions $k(x_1,\cdot\,), \ldots, k(x_n, \cdot\,)$. Since orthogonal projections decrease norms, we have $\bs{h}^TK_\ff^{-1}\bs{h} \le \|h\|^2_\bbH$. 

For the second term in \eqref{eq: piet} we note that
\[ \tr(Q_n^{-1}-K_n^{-1}) = \tr(Q_n^{-1}(K_n-Q_n)K_n^{-1}) \leq \|Q_n^{-1}\|\|K_n^{-1}\|\tr(K_n-Q_n), \]
where the matrix norms appearing on the right are both bounded by $\sigma^{-2}$.

Together we get
\begin{equation*}
\frac{1}{2} \E_{\bs{y}} \bs{y}^T(Q_n^{-1}-K_n^{-1})\bs{y}
\leq \frac{1}{\sigma^2}\Big( \|K_n- Q_n\| \|h\|_\bbH^2 + \sum_{i=1}^n(f_0(x_i)-h(x_i))^2 + \frac{1}{2} \tr(K_n-Q_n)\Big).
\end{equation*}
Combining this with \eqref{eq: klbound}, taking expectations over $\bs{x}$ and recalling that $K_n-Q_n = K_\ff - Q_\ff$, we arrive at the statement of the lemma.
\end{proof}

In the next section, we present two choices of inducing variables, also considered in  \cite{burt2019}, to which we apply Theorem \ref{thm: main}.

\section{Inducing variables from eigendecompositions}\label{sec:ind_var}

The covariance operator $T_k$ on $L^2(\calX,G)$ associated with the kernel $k$ is defined as
\begin{equation}
 T_k\psi (y) = \int_\calX k(x,y) \psi(x) \,dG(x). \label{eq: kernel}
\end{equation}
Note that this definition depends on the distribution $G$ of the design points. Since $k$ is a covariance kernel, the operator $T_k$ is positive (meaning $\ip{T_k\psi,\psi} \geq 0$ for all $\psi \in L^2(\calX,G)$). We assume that $k \in L^\infty(G \times G)$. One of the assertions of Mercer's Theorem (see e.g. \citeb{koenig1986}) is that consequently, $T_k$ is a Hilbert-Schmidt operator, and thus compact. It follows that $T_k$ has eigenvalues $\lambda_1 \geq \lambda_2 \geq \cdots \to 0$.

The covariance kernels used in practice satisfy these mild assumptions. We focus on three such kernels in this paper: the Mat\'ern kernel, the squared exponential kernel, and the kernel of a random series prior. For each kernel we consider one or two choices of inducing variables and discuss the conditions of Theorem \ref{thm: main}: we study the concentration function inequality \eqref{eq: con} and analyse the expected norm and trace terms \eqref{eq:norm} and \eqref{eq:trace}. The latter is done with help of the eigenvalues of the operator $T_k$. We consider kernels whose associated operator $T_k$ has exponentially or polynomially decreasing eigenvalues, that is, for $j=1,2,\ldots,$ we assume 
one of the conditions
\begin{align}
\lambda_j & \leq C_{\exp} b_n e^{-D_{\exp} b_nj}, \label{def:exp} \\
C_\alpha^{-1} j^{-1-2\alpha/d} \leq \lambda_j &  \leq C_\alpha j^{-1-2\alpha/d}, \label{def:poly}
\end{align}
for $0<b_n\leq 1$ and positive constants $C_{\exp},D_{\exp},C_\alpha$.

\subsection{Using the eigendecomposition of the covariance matrix}\label{sec:empeigen}

In this case we construct inducing variables using the  $m$ largest eigenvalues and the corresponding eigenvectors 
of the matrix $K_\ff = [k(x_i,x_j)]_{1 \leq i,j \leq n}$. 
We define
\begin{equation}\label{eq:indvar1}
u_j=\bs{v}_j^T \bs{f} = \sum_{i=1}^n v_j^i f(x_i),\qquad j=1,\ldots,m,
\end{equation}
where $\bs{v}_j=(v_j^1,v_j^2,\ldots,v_j^n)$ is the eigenvector corresponding to the $j$th largest eigenvalue $\mu_j$ of the matrix $K_\ff$. Note that each $u_j$ is a linear functional of $f$, and more precisely a linear combination of the values of $f$ evaluated at the observations $\bs{x}$. 
It is easy to verify (see also  Section C.1. of \citeb{burt2019}) that in this case we have 
\begin{gather*}
(K_{\bs{uu}})_{ij} = \cov_\Pi(u_i,u_j) = \mu_j \delta_{ij},\\
(K_{\bs{fu}})_{ij} = \cov_\Pi(f(x_i),u_j) = \mu_j v_j^i.
\end{gather*}
Hence in view of the identity $K_\ff = \sum_{j=1}^n \mu_j \bs{v}_j\bs{v}_j^T$,
\begin{gather}
Q_\ff = K_{\bs{fu}} K_{\bs{uu}}^{-1} K_{\bs{uf}} = \sum_{j=1}^m \mu_j \bs{v}_j\bs{v}_j^T,\nonumber \\
K_\ff - Q_\ff =\sum_{j=m+1}^n \mu_j \bs{v}_j\bs{v}_j^T.\label{ineq:eigen}
\end{gather}
Note that with this choice of $\bs{u}$ the matrix $Q_\ff$ is the optimal rank-$m$ approximation of $K_\ff$. 
The computational complexity of obtaining the first $m$ eigenvalues and the corresponding eigenvectors of $K_\ff$ numerically is $O(mn^2)$, by using for instance the Lanczos iteration (\citeb{lanczos1950}). Analytical expressions for the eigenvalues and eigenvectors of $K_\ff$ are not available for the majority of commonly used kernels.

Since the eigenvectors $\bs{v}_j$ are orthogonal, 
\begin{align}
\|K_\ff-Q_\ff\| &= \mu_{m+1}, \label{eq:met1norm} \\
\tr(K_\ff-Q_\ff) &= \sum_{j=m+1}^n \mu_j. \label{eq:met1trace}
\end{align}
%
%
%
\cite{burt2020} explain that this choice of $Q_\ff$ is the minimiser of both these quantities. As such, the right-hand sides of the above identities serve as benchmarks for other choices of inducing variables. To bound these, we will use repeatedly the part of Proposition 2 in \cite{shawe:2003} stating that
\begin{equation}\label{eq: shawe}
\E_{\bs{x}} \sum_{j=j_0}^n \mu_j /n \leq \sum_{j=j_0}^\infty \lambda_j
\end{equation}
for all $j_0$ between $1$ and $n$. 

We bound the expected trace and norm terms in Theorem \ref{thm: main}. For exponentially decreasing eigenvalues \eqref{def:exp} this is straightforward. Indeed, from \eqref{eq:met1trace} and \eqref{eq: shawe} we obtain
\begin{equation}\label{eq:expbounds}
\E_{\bs{x}} \|K_\ff-Q_\ff\| \leq \E_{\bs{x}} \tr(K_\ff-Q_\ff) \leq n \sum_{j=m+1}^n \lambda_j 
 \lesssim n \sum_{j=m+1}^\infty b_n e^{-D_{\exp} b_n j} \lesssim n e^{-D_{\exp}b_n m},
\end{equation}
which suffices for our purposes. Polynomially decaying eigenvalues require more work as we need to do better than bounding the operator norm by the trace.

\begin{lemma}\label{lem: UB:emp:eigenvalue}
If the eigenvalues $\lambda_1,\lambda_2,\ldots$ of the operator \eqref{eq: kernel} are polynomially decaying \eqref{def:poly}, then there is a constant $\bar{C}_\alpha$ such that
\begin{gather*}
\E_{\bs{x}} \|K_\ff - Q_\ff\| \leq \bar{C}_\alpha n m^{-1-2\alpha/d}, \\
\E_{\bs{x}} \tr(K_\ff-Q_\ff) \leq \bar{C}_\alpha nm^{-2\alpha/d}, 
\end{gather*}
for any $2 \leq m \leq n$.
\end{lemma}

\begin{proof}
We deal with the norm term using \eqref{eq:met1norm}. We argue by contradiction. Suppose that for all $i\in\{m/2,\ldots,m\}$ we have $\E_{\bs{x}} \mu_i/n > \tilde{C}_\alpha\lambda_i$, where $\tilde{C}_\alpha = 1+ dC_\alpha^2/\alpha$. Since
\begin{equation}\label{eq: polytr}
\sum_{i=m+1}^{\infty} \lambda_i \leq C_\alpha \sum_{i=m+1}^{\infty} i^{-1-2\alpha/d}\leq C_\alpha\int_{m}^{\infty} t^{-1-2\alpha/d} dt = \frac{C_\alpha d}{2\alpha} m^{-2\alpha/d}, 
\end{equation}
\[
\sum_{i=m/2}^{m} \lambda_i \geq C_\alpha^{-1}\sum_{i=m/2}^{m}i^{-1-2\alpha/d}\geq (2C_\alpha)^{-1} m^{-2\alpha/d},
\]
we have
\[
\E_{\bs{x}} \sum_{i=m/2}^n\mu_i/n \geq \E_{\bs{x}} \sum_{i=m/2}^m\mu_i/n > \tilde{C}_\alpha \sum_{i=m/2}^{m} \lambda_i\geq \sum_{i=m/2}^\infty \lambda_i,
\]
but this contradicts \eqref{eq: shawe}. Therefore there exists $i\in\{m/2,\ldots,m\}$ such that $\E_{\bs{x}}\mu_i/n \leq \tilde{C}_\alpha\lambda_i$. Hence
\[
\E_{\bs{x}} \mu_{m+1}\leq  \E_{\bs{x}}\mu_i\leq n\tilde{C}_\alpha\lambda_i\leq n\tilde{C}_\alpha \lambda_{m/2}\leq (\tilde{C}_\alpha C_\alpha 2^{1+2\alpha/d})m^{-1-2\alpha/d}n,
\]
hence, recalling \eqref{eq:met1norm}, we obtain the bound on the expected norm term.

Regarding the trace term, the inequality \eqref{eq: shawe} implies
\[ \E_{\bs{x}} \tr(K_\ff-Q_\ff) = \E_{\bs{x}} \sum_{i=m+1}^n \mu_i \leq n\sum_{i=m+1}^\infty \lambda_i. \]
The inequality regarding the trace in the statement of the lemma then follows immediately using \eqref{eq: polytr}.
\end{proof}

We turn to the second choice of inducing variables before applying these results to the chosen kernels.

\subsection{Using the eigendecomposition of the covariance operator}\label{sec:opeigen}

The previous method requires computing the eigenvalues and the eigenvectors of the matrix $K_\ff$, which for large data sets becomes computationally demanding. Another choice of inducing variables is
\begin{equation}\label{eq:indvar2}
u_j=\int_{\mathcal{X}} f(x)\varphi_j(x) \, dG(x), \qquad j =1, \ldots, m, 
\end{equation}
where $\varphi_1,\varphi_2,\ldots$ are the eigenfunctions of the kernel operator $T_k$, corresponding to the eigenvalues $\lambda_1,\lambda_2,\ldots,$ so $\int k(x,y) \varphi_i(x) \,dG(x) = \lambda_i \varphi_i(y)$. In case $\mathcal X$ is a compact interval and the functions $\varphi_j$ form a Fourier series, this choice of inducing variables yields the variational Fourier features described in \cite{hensman2018}.

The relevant covariance matrices for the inducing variables \eqref{eq:indvar2} are
\begin{gather*}
(K_{\bs{uu}})_{ij} = \cov_\Pi(u_i,u_j) = \lambda_j \delta_{ij},\\
(K_{\bs{fu}})_{ij} = \cov_\Pi(f(x_i),u_j) = \lambda_j \varphi_j(x_i)
\end{gather*}
(see again Appendix C of \citeb{burt2019} for the proof of these statements). Then in view of Mercer's theorem, $K_\ff = \sum_{j=1}^\infty \lambda_j \bs{\varphi}_j\bs{\varphi}_j^T$ where we denote $\bs{\varphi}_j=(\varphi_j(x_1),\ldots,\varphi_j(x_n))$, so
\begin{gather*}
Q_\ff=\sum_{j=1}^m \lambda_j \bs{\varphi}_j\bs{\varphi}_j^T, \\
K_\ff-Q_\ff= \sum_{j=m+1}^\infty \lambda_j\bs{\varphi}_j\bs{\varphi}_j^T.
\end{gather*}
(Note that unlike the $\bs{v}_j$ from the previous section, the vectors $\bs{\varphi}_j$ do not necessarily form an orthonormal basis of $\mathbb{R}^n$.)

With this choice of inducing variables, we obtain for the expected trace term
\begin{equation}\label{eq: method2trace}
\E_x \tr(K_\ff-Q_\ff) = \sum_{j=m+1}^\infty \lambda_j \sum_{i=1}^n \E_{\bs{x}} \varphi_j(x_i)^2 = n\sum_{j=m+1}^\infty \lambda_j.
\end{equation}
This is exactly the upper bound we obtained for the trace term in the previous section.
For the exponentially decaying eigenvalues, we bound the operator norm just as in \eqref{eq:expbounds} by 
\begin{equation}
\E_{\bs{x}} \|K_\ff-Q_\ff\| \leq \E_{\bs{x}} \tr(K_\ff-Q_\ff) \lesssim n e^{-D_{\exp}b_n m}.
\end{equation}
The results regarding the polynomially decreasing eigenvalues are summarized in the next lemma.

%
%
%
%
\begin{lemma}\label{lem: method2:spectral}
Assume that the eigenvalues  $\lambda_1,\lambda_2,\ldots$ of the operator \eqref{eq: kernel} are polynomially decaying \eqref{def:poly}, with $\alpha>d$. Suppose the corresponding eigenfunctions of the operator $T_k$ are uniformly bounded. Then
\begin{gather*}
\E_{\bs{x}} \| K_\ff-Q_\ff\|\lesssim 1+nm^{-1-2\alpha/d}+n^{d/(2\alpha)}m^{-2\alpha/d}\log n, \\
\E_{\bs{x}} \tr(K_\ff-Q_\ff) \leq \frac{C_\alpha d}{2\alpha} nm^{-2\alpha/d}. 
\end{gather*}
\end{lemma}

\begin{proof}
The expected trace inequality follows upon combining \eqref{eq: method2trace} and \eqref{eq: polytr}. The expectation of the spectral norm is bounded by distributing over the event
\begin{align*}
A_n(C)=\{\bs{x}\in \mathcal{X}^n:\,|\langle\bs{\varphi}_j, \bs{\varphi}_k\rangle-n\delta_{jk}|\leq C\sqrt{n\log n}, \quad m < j,k\leq n^{d/(2\alpha)} \},
\end{align*}
and its complement.

In view of Lemma \ref{lem: emp:basis} below, there exists a large enough $C>0$ such that $\P_{\bs x} (A_n(C)^c)\leq n^{-1}$. Using the crude estimate
\[ \|K_\ff-Q_\ff\| \leq \tr(K_\ff-Q_\ff) = \sum_{j=m+1}^\infty \sum_{i=1}^n \lambda_j \varphi_j(x_i)^2 \leq n C_\varphi \sum_{j=1}^\infty \lambda_j \lesssim n \]
(the constant $C_\varphi$ being the uniform bound for the $\varphi_j$) we then obtain
\[ \E_{\bs x} \bs 1_{A_n(C)^c} \|K_\ff-Q_\ff\| \leq n \P_{\bs x}(A_n(C)^c) \lesssim 1. \]

On the event $A_n(C)$ we use 
\begin{multline*}
\E_{\bs x} \bs1_{A_n(C)} \| K_\ff-Q_\ff \| \leq \E_{\bs x}\bs1_{A_n(C)} \Big\| \sum_{k=m+1}^{n^{d/(2\alpha)} }\lambda_k \bs{\varphi}_k\bs{\varphi}_k^T\Big\| + \E_{\bs x} \Big\| \sum_{k>n^{d/(2\alpha)}}\lambda_k \bs{\varphi}_k\bs{\varphi}_k^T\Big\| \\
\leq \E_{\bs x} \bs1_{A_n(C)} \max_{\|v\|_2=1}v^{T}\Big(\sum_{k=m+1}^{n^{d/(2\alpha)}}\lambda_k \bs{\varphi}_k\bs{\varphi}_k^T\Big) v+\E_{\bs x} \tr\Big(\sum_{k>n^{d/(2\alpha)}}\lambda_k \bs{\varphi}_k\bs{\varphi}_k^T\Big),
\end{multline*}
where the last inequality follows from the positive semi-definiteness of the matrices $\lambda_k \bs{\varphi}_k\bs{\varphi}_k^T$. The second bounding term equals
\[
\tr\Big(\sum_{k>n^{d/(2\alpha)}}\lambda_k \E_{\bs{x}}\bs{\varphi}_k\bs{\varphi}_k^T\Big) = n\sum_{k>n^{d/(2\alpha)}}\lambda_k\lesssim n \sum_{k>n^{d/(2\alpha)}}k^{-1-2\alpha/d} \lesssim 1.
\]
Lastly, we deal with the first term by bounding
\[
 \max_{\|v\|_2=1}v^{T}\Big(\sum_{k=m+1}^{n^{d/(2\alpha)} }\lambda_k \bs{\varphi}_k\bs{\varphi}_k^T\Big) v
=  \max_{\|v\|_2=1} \sum_{k=m+1}^{n^{d/(2\alpha)} }\lambda_k \langle v,\bs{\varphi}_k \rangle^2.
\]
on the event $A_n(C)$. It is sufficient to consider vectors $v$ of the form $v=\sum_{k=m+1}^{n^{d/(2\alpha)}} \rho_k \bs{\varphi}_k$. On the event $A_n(C)$, using that $\alpha>d$,
\begin{align*}
1&= \|v\|_2^2= \sum_{k,j=m+1}^{n^{d/(2\alpha)}} \rho_j\rho_k \langle   \bs{\varphi}_j, \bs{\varphi}_k\rangle\geq  \sum_{k,j=m+1}^{n^{d/(2\alpha)}} \rho_j\rho_k \Big(n\delta_{jk}- C\sqrt{n\log n}\Big)\\
& \geq 
\sum_{k=m+1}^{n^{d/(2\alpha)}}\rho_k^2 \Big( n -n^{d/(2\alpha)}C\sqrt{n\log n} \Big)\geq \frac{n}{2}\sum_{k=m+1}^{n^{d/(2\alpha)}}\rho_k^2,
\end{align*} 
and therefore
\begin{align*}
\max_{\|v\|_2=1} \sum_{k=m+1}^{n^{d/(2\alpha)}}\lambda_k \langle v,\bs{\varphi}_k \rangle^2
&=  \max_{\|v\|_2=1} \sum_{k=m+1}^{n^{d/(2\alpha)}} \lambda_k  \Big( \sum_{j=m+1}^{n^{d/(2\alpha)}} \rho_j  \langle \bs{\varphi}_j,\bs{\varphi}_k \rangle \Big)^2\\
&\leq  \max_{\|v\|_2=1} \sum_{k=m+1}^{n^{d/(2\alpha)}} \lambda_k  \Big( \sum_{j=m+1}^{n^{d/(2\alpha)}} |\rho_j|  (n\delta_{jk} +C\sqrt{n\log n}) \Big)^2\\
&\lesssim   \max_{\|v\|_2=1} \sum_{k=m+1}^{n^{d/(2\alpha)}} \lambda_k  \Big( n^2\rho_k^2 +n^{\frac{d+2\alpha}{2\alpha}} \log n\sum_{j=m+1}^{n^{d/(2\alpha)}} \rho_j^2 \Big)\\
&\lesssim n\lambda_{m+1}  \big(\max_{\|v\|_2=1}  n\sum_{k=m+1}^{n^{d/(2\alpha)}} \rho_k^2\big) +n^{d/(2\alpha)}\log n \sum_{k=m+1}^{n^{d/(2\alpha)}} \lambda_k\\
&\lesssim nm^{-1-2\alpha/d}+n^{d/(2\alpha)}m^{-2\alpha/d}\log n. 
\end{align*}
The proof is concluded by multiplying with $\bs 1_{A_n(C)}$ and taking expectations $\E_{\bs x}$ in the above display.
\end{proof}

The following lemma provides the concentration inequality for the empirical inner product of the eigenfunctions, used in the proof of the preceding lemma.

\begin{lemma}\label{lem: emp:basis}
For orthonormal functions $\varphi_1,\varphi_2,\ldots,\varphi_{M_n}$ w.r.t. the measure $G$ such that $|\varphi_i|\leq C_{\varphi}$ on $\mathcal{X}$ and $\bs x=( x_1,x_2,\ldots,x_n)$ i.i.d. with common distribution $G$, the random vectors $\bs\varphi_\ell = (\varphi_\ell(x_1),\ldots,\varphi_\ell(x_n))$ satisfy
\begin{align*}
\P_{\bs x}\Big( \sup_{1\leq \ell,k\leq M_n} | \langle\bs\varphi_\ell,\bs\varphi_k \rangle -n\delta_{\ell k}|\geq C\sqrt{n\log n}\Big)\leq M_n^2 n^{-(C/C_{\varphi}^2)^2/2}
\end{align*}
for any $C>0$.
\end{lemma}
\begin{proof}
By the subadditivity of the probability and using Hoeffding's inequality for bounded random variables we get that
\begin{align*}
\P_{\bs x}( \sup_{1\leq \ell,k\leq M_n} &| \langle\bs\varphi_\ell,\bs\varphi_k \rangle -n\delta_{\ell k}|\geq C\sqrt{n\log n})\\
&\leq M_n^2 \sup_{\ell,k} \P_{\bs x}(  | n^{-1}\langle\bs\varphi_\ell,\bs\varphi_k \rangle -\delta_{\ell k}|\geq C\sqrt{n^{-1}\log n})\\
&\leq M_n^2 \exp\{ -\frac{2n^2 C^2n^{-1}\log n }{n4C^4_{\varphi}}\}=M_n^2 \exp\{ -  \frac{C^2}{2C_{\varphi}^4}\log n \},
\end{align*}
finishing the proof of the statement.
\end{proof}

\section{Concrete examples}\label{sec:examples}
%
%
%
%

We consider three explicit examples to demonstrate how the approximation theory from the previous section can be used to apply the main theorem in Section \ref{sec:main}. 
The contraction rates we obtain depend on the smoothness properties of the underlying true 
regression function $f_0$. To make this precise, we recall the definition of two smoothness classes.

The Hölder space $C^\alpha(\calX)$ of smoothness $\alpha > 0$ consists of those functions on $\calX$ with Hölder regularity $\alpha$. This means partial derivatives of order up to $\alpha_0 := \lceil\alpha\rceil - 1$ exist and are uniformly bounded, and derivatives of order equal to $\alpha_0$ satisfy a Hölder condition with exponent $\alpha - \alpha_0$.

The Sobolev space $H^\alpha(\calX)$ is the collection of restrictions $f_0|_\calX$ to $\calX$ of functions $f_0 : \bbR^d \to \bbR$ with Fourier transform $\hat f_0(\lambda) = (2\pi)^{-d} \int_{\bbR^d} e^{i\ip{\lambda,x}} f_0(x) \,dx$ satisfying
\[ \int (1+\|\lambda\|^2)^\alpha |\hat f_0(\lambda)|^2 \,d\lambda < \infty. \]
For $\alpha\in\mathbb{N}$ the space $H^\alpha(\calX)$ coincides with the space of functions with square integrable weak $\alpha$-derivatives over $\calX$.

\subsection{Matérn kernel}\label{s:mat}
The Matérn prior is the centered GP whose covariance kernel is
\begin{equation}\label{eq:mat}
k(x,y) = c_1 \|x-y\|^\alpha K_\alpha(c_2 \|x-y\|),
\end{equation}
where $c_1,c_2,\alpha$ are positive constants and $K_\alpha$ is the modified Bessel function of the second kind (see \citeb{Rasmussen2006}). If $\calX = [0,1]^d$ and $f_0 \in C^\alpha(\calX) \cap H^\alpha(\calX)$, then it is known that the true posterior 
contracts around $f_0$ at the rate $n^{-\alpha/(d+2\alpha)}$; see e.g. \cite{vaart2011}. 
This is the optimal minimax rate of contraction for this problem. 
The following corollary asserts that if the number of inducing variables is chosen at least of 
the order $n^{d/(d+2\alpha)}$, then for the first class of inducing variables considered above, 
the variational posterior attains this optimal rate as well.

\begin{corollary}\label{c:matern}
Let $k$ be the Matérn kernel \eqref{eq:mat} on $\calX = [0,1]^d$ and let $G$ be a distribution with bounded Lebesgue density. Suppose that the inducing variables \eqref{eq:indvar1} are used and $\alpha > d/2$. Then the variational posterior contracts around $f_0 \in C^\alpha(\calX) \cap H^\alpha(\calX)$ at the rate $\epsilon_n = n^{-\alpha/(d+2\alpha)}$ for $m=m_n \geq n^{d/(d+2\alpha)}$.
\end{corollary}
\begin{proof}
It follows from the assumptions on $f_0$ and $\alpha$ combined with (results leading to) Theorem 5 in \cite{vaart2011} that $\varphi(\epsilon) \lesssim \epsilon^{-d/\alpha}$, so the concentration function inequality \eqref{eq: con} holds for $\epsilon_n$ as specified. 

The assumptions on $G$ allow for an application of Theorem 1 in \cite{seeger2007}, whose proof yields \eqref{def:poly} for the eigenvalues of the kernel operator $T_k$. Lemma \ref{lem: UB:emp:eigenvalue} implies that the trace and norm inequalities in Theorem \ref{thm: main} hold for $m$ as given. This yields the contraction statement for the variational posterior.
\end{proof}

The other choice of inducing variables \eqref{eq:indvar2} is not considered here, since for the 
stationary Mat{\'e}rn process we don't have access 
to the eigenfunctions of the kernel operator $T_k$.  
For $G$ equal to the uniform distribution on $[0,1]^d$, finding the eigenfunctions and eigenvalues is 
equivalent to finding the Karhunen-Lo\`eve expansion. 
Explicit expressions appear only to be available for the case $\alpha = 1/2$ of the Ornstein-Uhlenbeck process, 
see for instance  \cite{corlay2015}.

%

\subsection{Squared exponential kernel}
The squared exponential process on $\calX = \bbR^d$ with length scale $b > 0$ 
is the centered GP on $\bbR^d$ with covariance function 
\begin{equation}\label{eq:sqexp}
k(x,y) = \exp(-\|x-y\|^2/b^2).
\end{equation} 
The structure of the RKHS and sharp bounds for the concentration function 
are known for this process, but in existing results the process is usually viewed
on a compact subset of $\bbR^d$ and the concentration function 
relative to the uniform norm is considered, see for instance 
\cite{vaart2009}, \cite{vaart2011}. 
In this paper we  want to consider the example that $G$ is a normal 
distribution, in which case the existing results do not directly apply. 
Therefore we adapt the relevant results,
 viewing the squared exponential 
process  as a random element in the space $L^2(\calX,G)$. 

We formulate the following lemma for slightly more general distributions $G$ with sub-Gaussian tails, that is, 
we assume that there exist constants $C_1, C_2 > 0$ such that 
\begin{equation}\label{eq: sg}
G(x: \|x\|> a) \le C_1e^{-C_2 a^2}
\end{equation}
for all $a>0$ large enough.
It is seen from the proof that the statement of the lemma can easily be adapted to 
cases with different tail behaviours. 

\begin{lemma}\label{lem: conc_function}
Let $k$ be the squared exponential kernel \eqref{eq:sqexp} with length scale $b = b_n = n^{-1/(d+2\alpha)}$. Suppose that $f_0 \in C^\alpha(\bbR^d) \cap L^2(\bbR^d)$, and $G$ satisfies the sub-Gaussian tail bound \eqref{eq: sg} on $\calX=\bbR^d$. Then the concentration function inequality \eqref{eq: con} is satisfied for $\eps_n$ a multiple of $n^{-\alpha/(d+2\alpha)}\log^{\kappa/2} n$, where $\kappa = 1+3d/2$.
\end{lemma}

\begin{proof}
This follows from combining Lemma \ref{lem: smallball} and \ref{lem: decentering} in Appendix \ref{a:sqexp}.
\end{proof}

By the results of \cite{vaart2009}, under the assumptions of the above lemma,
the true posterior contracts around $f_0$ at the optimal rate $n^{-\alpha/(d+2\alpha)}$, up to a logarithmic factor.
The following corollary asserts that if $d=1$ and $G$ is a normal distribution, the same is true 
for the variational posteriors considered above.

\begin{corollary}\label{cor:sqexp}
Let $k$ be the squared exponential kernel \eqref{eq:sqexp} with $b=b_n=n^{-1/(1+2\alpha)}$, and $G$ a centered Gaussian distribution on $\calX=\bbR$. Then the variational posterior using either choice of inducing variables \eqref{eq:indvar1} or \eqref{eq:indvar2} contracts around $f_0 \in C^\alpha(\bbR) \cap L^2(\bbR)$ at the rate $\epsilon_n = n^{-\alpha/(1+2\alpha)} (\log n)^{5/4}$, provided that  $m = m_n  \ge D_{\exp}^{-1} n^{1/(1+2\alpha)} \log n$.
\end{corollary}
\begin{proof}
In Lemma \ref{lem: conc_function} we have already established that the concentration function inequality is satisfied for the specified truth $f_0$, scale $b_n$, and rate $\epsilon_n$.

We now prove the eigenvalues of the covariance operator are exponentially decaying. For notational convenience suppose that $a>0$ is such that $G$ has density $p(x) \propto e^{-2ax^2}$. There is an explicit expression for the eigenvalues (see \citeb{Rasmussen2006})
\[ \lambda_j = \sqrt{2a/A_n} \Big(\frac{1}{A_nb_n^2}\Big)^{j-1}, \quad j=1,2,\ldots \]
with $A_n = a+b_n^{-2}+\sqrt{a^2+2ab_n^{-2}}$. We note that
\[ \frac{1}{A_nb_n^2} = 1-z_n \leq e^{-z_n} \]
for $z_n = \sqrt{a^2b_n^4+2ab_n^2}-ab_n^2$, and $z_n/b_n \to \sqrt{2a}$ as $n\to\infty$, so $z_n > D_{\exp} b_n$ when $0< D_{\exp} < \sqrt{2a}$ and $n$ is sufficiently large. Then
\[ \lambda_j \leq \sqrt{2a/A_n} e^{-z_n j} \lesssim b_n e^{-D_{\exp}b_n j}, \]
so we are in the situation of \eqref{def:exp}. By \eqref{eq:expbounds} and \eqref{eq: method2trace}, the choice of $m$ yields 
\[ \E_{\bs{x}} \|K_\ff-Q_\ff\| \leq \E_{\bs{x}} \tr(K_\ff-Q_\ff) \lesssim n e^{-D_{\exp}b_n m} \leq 1, \]
so the conditions of Theorem \ref{thm: main} are satisfied. 
\end{proof}

\begin{remark}
A stronger requirement on the smoothness of $f_0$ is that it belongs to the RKHS $\bbH$ associated to the prior. 
In this case the RKHS approximation term in the concentration function is bounded by a constant, so the contraction rate is characterised by the small ball probability which is bounded in Lemma \ref{lem: smallball}. One can take a fixed length scale $b>0$, so that the concentration function inequality holds for $\epsilon_n$ satisfying
\[ \Big(\log\frac{1}{\eps_n}\Big)^\kappa \lesssim n\eps_n^2. \]
This is fulfilled by the rate $\epsilon_n = n^{-1/2} (\log n)^{\kappa/2}$, which is almost the parametric rate $n^{-1/2}$. By the arguments used to establish the above corollary, the variational posterior contracts at this rate when $m_n$ is taken of the order $\log n$. This is also what \cite{burt2019} suggest for the exponential kernel. Our Corollary \ref{cor:sqexp} illustrates that this choice may not be optimal if $f_0$ is not so smooth that it  belongs to the RKHS of the covariance kernel, which only contains analytic functions. 
See also the numerical illustration in Section \ref{sec:num}.
\end{remark}

\subsection{Random series prior}

The last choice of kernel is one defined through a series expansion. 
We take $\calX=[0,1]^d$ and consider a uniform distribution $G$ for the design points. Let $(\varphi_j)$ be an orthonormal basis of the corresponding function space $L^2[0,1]^d$. Suppose that the basis functions are continuous and uniformly bounded, that is, $\sup_j \sup_x |\varphi_j(x)|<\infty$. Define, for  $\alpha > 0$, the series
\begin{equation}\label{eq:kh}
f(x) = \sum_{j=1}^\infty j^{-1/2-\alpha/d} \varphi_j(x) Z_j, \qquad x \in [0,1]^d,
\end{equation}
where $(Z_j)$ is a sequence of i.i.d. standard normal random variables.  The series converges uniformly and the resulting process $(f(x):x \in [0,1]^d)$ is a centered GP with covariance function
\begin{equation}\label{eq:serieskernel}
k(x,y) = \sum_{j=1}^\infty j^{-1-2\alpha/d} \varphi_j(x)\varphi_j(y).
\end{equation}
By construction, $(\varphi_j)$ is the orthonormal eigenbasis of the associated operator $T_k$ with eigenvalues $\lambda_j=j^{-1-2\alpha/d}$. We note that one can generalise these priors to compact Riemannian manifolds $\mathcal X$ (in fact, the compactness assumption can also be relaxed for appropriate choice of $G$) and the coefficients $j^{-1/2-\alpha/d}$ can be replaced by any sequence $\sqrt{\lambda_j}$ such that \eqref{eq:kh} converges.

We consider contraction of the variational posterior corresponding to this prior. A function $f_0 \in L^2[0,1]^d$ has the expansion $f_0 = \sum_{j=1}^\infty f_{0,j} \varphi_j$ where $f_{0,j} =\ip{f_0,\varphi_j}$. Here we consider functions in the Sobolev space
\[ 
\tilde H^\alpha = \{ f \in L^2[0,1]^d : \|f\|_\alpha < \infty\}, \qquad \|f\|_\alpha^2 = \sum_j j^{2\alpha/d} |\ip{f,\varphi_j}|^2. 
\]
In general this space is different from the previously defined $H^\alpha([0,1]^d)$ since it depends on the choice of basis functions $\varphi_j$. 
If $(\varphi_j)$ is the standard Fourier basis in $d=1$, however, the spaces coincide.

With either choice of inducing variables discussed earlier, the variational posterior contracts around elements of $\tilde H^\alpha$ at the minimax rate. 

\begin{corollary}
Consider the kernel \eqref{eq:serieskernel} for some uniformly bounded orthonormal basis $(\varphi_j)$ of $L^2[0,1]$ consisting of continuous functions. Suppose that either
\begin{itemize}
\item[--] the inducing variables \eqref{eq:indvar1} are used and $\alpha > d/2$, or
\item[--] the inducing variables \eqref{eq:indvar2} are used and $\alpha>d$.
\end{itemize}
Then the variational posterior contracts around $f_0 \in \tilde H^\alpha$ at the rate $\epsilon_n = n^{-\alpha/(d+2\alpha)}$ for $m=m_n \geq n^{d/(d+2\alpha)}$.
\end{corollary}

\begin{proof}
We start by bounding the concentration function. By Theorem 4.1 in \cite{rkhs}, the function $h := \sum_{j=1}^J \ip{f_0,\varphi_j}\varphi_j = \sum_{j=1}^J f_{0,j} \varphi_j$ is an element of the RKHS $\bbH$ of the prior with squared norm $\|h\|_\bbH^2 = \sum_{j=1}^J |f_{0,j}|^2/\lambda_j$. If $f_0 \in \tilde H^\alpha$ then we have 
\[ \|h\|_\bbH^2 = \sum_{j=1}^J |f_{0,j}|^2 j^{1+2\alpha/d} \leq J \|f_0\|_\alpha^2. \]
Moreover,
\[ \|f_0-h\|_{2,G}^2 = \sum_{j>J} |f_{0,j}|^2 \leq \|f_0\|_\alpha^2 J^{-2\alpha/d} \]
so by choosing $J$ of the order $\epsilon^{-d/\alpha}$, it follows that
\[ \inf_{h\in \bbH:\|h-f_0\|_{2, G} \le \eps} \|h\|^2_\bbH \lesssim \eps^{-d/\alpha}. \]
By the expansion \eqref{eq:kh}, the centered small ball probability can be written as
\[ \Pi(f : \|f\|_{2,G} \leq \epsilon) = \Pr\Big( \sum_{j=1}^\infty j^{-1-2\alpha/d} Z_j^2 \leq \epsilon^2\Big). \] 
By Corollary 4.3 in \cite{dunker1998},
\[ - \log \Pr\Big( \sum_{j=1}^\infty j^{-1-2\alpha/d} Z_j^2 \leq \epsilon^2\Big) \lesssim \epsilon^{-d/\alpha}. \]
It follows that the concentration function inequality \eqref{eq: con} holds for $\epsilon_n$ as specified (up to a constant), and this is the rate at which the true posterior contracts.

Evidently the eigenvalues satisfy \eqref{def:poly}. The trace and norm inequalities in Theorem \ref{thm: main} are readily verified for our choice of $m$ with the help of either Lemma \ref{lem: UB:emp:eigenvalue} or Lemma \ref{lem: method2:spectral}. This yields the contraction statement for the variational posterior.
\end{proof}

\section{Numerical experiments}\label{sec:num}

We illustrate the results of Section \ref{sec:ind_var} by two numerical experiments, varying both the kernel and the choice of inducing variables.

\subsection{Matérn kernel – method 1}

We simulate $n=3000$ samples $x_i \sim \textrm{uniform}[0,1]$ and $y_i \sim \mathcal N(f_0(x_i),\sigma^2)$ with $\sigma = 0.2$ and
\[ f_0(x) = |x-0.4|^\alpha - |x-0.2|^\alpha \]
for $\alpha = 0.6$, which is plotted in Figure \ref{fig:f0mat}. We use the Matérn-$\alpha$ kernel for the GP prior and study the variational posterior using the inducing variables obtained from the covariance matrix (Section \ref{sec:empeigen}). 

We compare the behaviour of the true and variational Bayes methods for different choices of the number of inducing points. Figures \ref{fig:goodmat} and \ref{fig:badmat} show the mean and pointwise 95\% credible regions (intervals centered vertically around the posterior mean which have posterior mass $0.95$) for both the true and variational posterior. According to Corollary \ref{c:matern}, $m$ should be at least $n^{1/(1+2\alpha)} \approx 40$. Figure \ref{fig:goodmat} illustrates this: here $m=40$, and although the variational posterior is in general a bit smoother, its credible region is hardly larger than that of the true posterior. On the contrary, one can conclude from Figure \ref{fig:badmat} that it is unwise to take a significantly lower number of inducing variables. The variational posterior mean is far too smooth and credible regions are too wide.

Table \ref{table:kl} shows estimates of the expected Kullback-Leibler divergence \eqref{eq: titsias}, computed from 100 repetitions of the above experiment for different $n$. We used $m=n^{1/(1+2\alpha)}$ inducing variables so that by Corollary \ref{c:matern} the variational posterior contracts at the minimax rate. Note that the KL-divergence increases with $n$, meaning that it does not vanish. This is in according with our theory, which says that the $\P_0$-expectation of the KL-divergence need only be of the order $n\epsilon_n^2 \to \infty$ (see the proof of Theorem \ref{thm: main}).

\begin{figure*}[!h]
	\makebox[\textwidth][c]{\includegraphics[scale=.8]{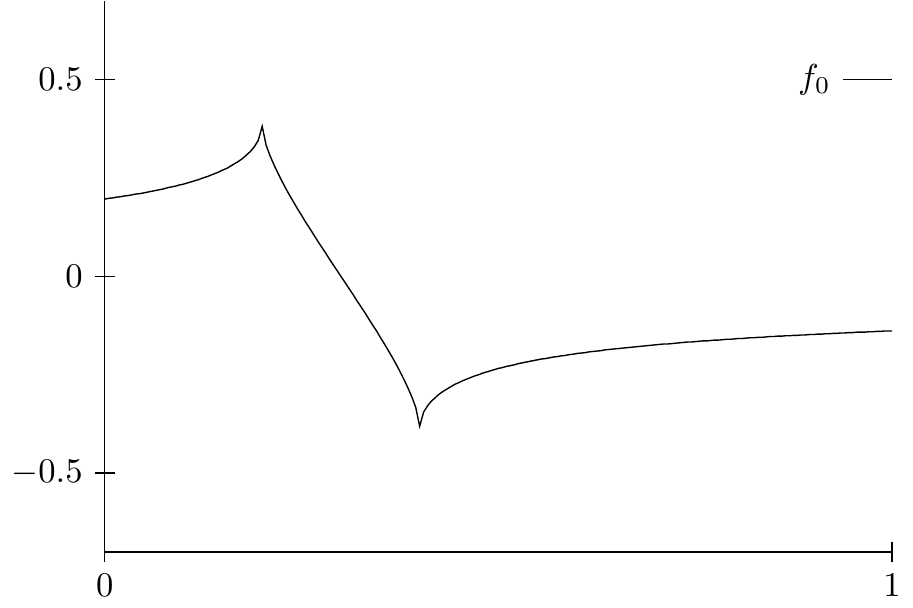}}
	\caption{plot of $f_0 = |x+1|^\alpha - |x+3/2|^\alpha$ for $\alpha = 0.8$}
	\label{fig:f0mat}
\end{figure*}

\begin{figure}[!h]
	\centering
	\begin{minipage}[b]{0.45\textwidth}
		\includegraphics[scale=.8]{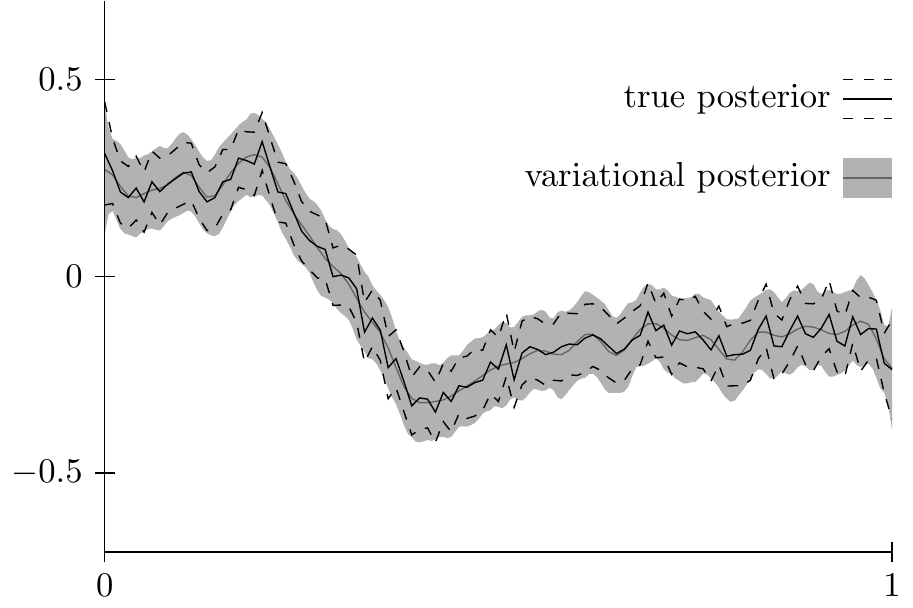}
		\caption{True and variational posterior and credible regions for Matérn prior and $m=40$ inducing variables from method 1}
		\label{fig:goodmat}
	\end{minipage}
	\hfill
	\begin{minipage}[b]{0.45\textwidth}
		\includegraphics[scale=.8]{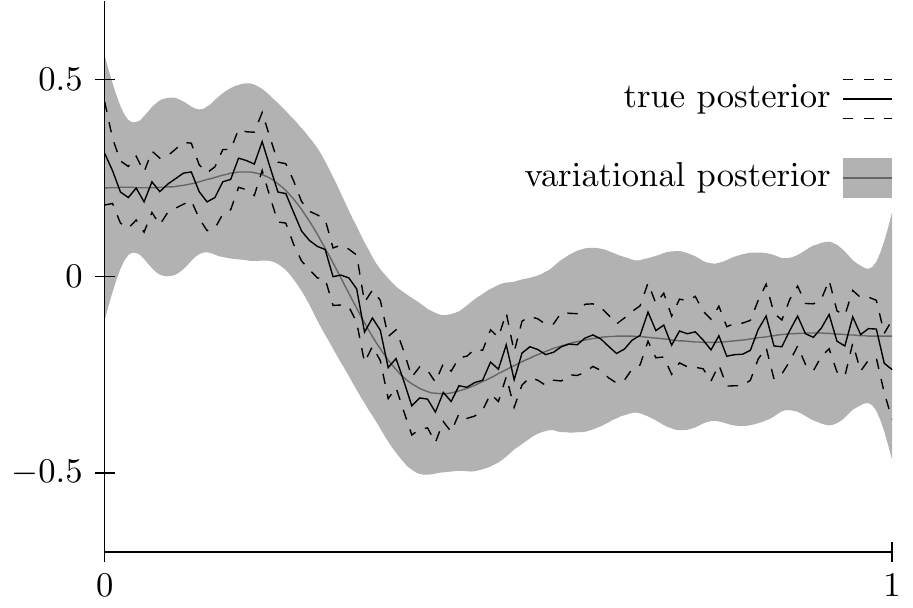}
		\caption{True and variational posterior and credible regions for Matérn prior and $m=10$ inducing variables from method 1}
		\label{fig:badmat}
	\end{minipage}
\end{figure}

\begin{table}[!h]\centering
\begin{tabular}{r|rr}
	$n$    & $\KL(\Psi(\,\cdot\mid\bs x,\bs y)\|\Pi(\,\cdot\mid\bs x,\bs y))$ \\ \hline
	$100$  & $14.71$ $(1.75)$ \\
	$300$  & $25.20$ $(2.31)$ \\
	$1000$ & $42.09$ $(3.23)$ \\
	$3000$ & $68.90$ $(3.94)$ \\
\end{tabular}
\caption{Estimates of the KL-divergence between variational and true posterior (average over 100 repeated experiments). Estimated standard deviations are given between brackets.}
\label{table:kl}
\end{table}

\subsection{Squared exponential kernel – method 2}

In a similar fashion, we simulate $n=5000$ samples $x_i \sim \mathcal N(0,1)$ and $y_i$ from the $\mathcal N(f_0(x_i),\sigma^2)$ distribution with
\[ f_0(x) = |x+1|^\alpha - |x+3/2|^\alpha \]
for $\alpha = 0.8$ and $\sigma = 0.2$. The function $f_0$ is plotted in Figure \ref{fig:f0}. Although strictly speaking $f_0\notin L^2(\calX,G)$, one can easily modify its tails maintaining $f_0 \in C^\alpha(\bbR)$ (also note that with high probability all $x_i$ are in a large compact set).

We use the squared exponential kernel as defined in \eqref{eq:sqexp} with $b=b_n = 4 n^{-1/(1+2\alpha)}$ and the variational Bayes method with operator eigenvectors (Section \ref{sec:opeigen}) as the inducing variables. 

Corollary \ref{cor:sqexp} prescribes that we take $m_n$ at least
\[ (D_{\exp} b_n)^{-1} \log n \approx \sqrt{2a} (n^{1/(1+2\alpha)}/4) \log n \approx 80, \]
where $a = 1/4$ to ensure that $G = \mathcal N(0,1)$. Figure \ref{fig:good} illustrates that this is indeed a good choice of $m$.  One can observe that the true and variational posterior are virtually indistinguishable, i.e., there is almost no loss of information in the variational Bayes method.

In Figure \ref{fig:bad} we take a smaller number $m=40$ of inducing points than the optimal $m\approx 80$. One can observe that in this case the variational posterior mean is overly smooth, although still gives a reasonable estimate of $f_0$. The main difference when considering insufficiently many inducing variables, is that the variational posterior overestimates variance. Here, too, the variational Bayes method provides overly conservative, way too large credible sets compared to the true posterior.

\vspace{1cm}

\begin{figure*}[!h]
	\makebox[\textwidth][c]{\includegraphics[scale=.8]{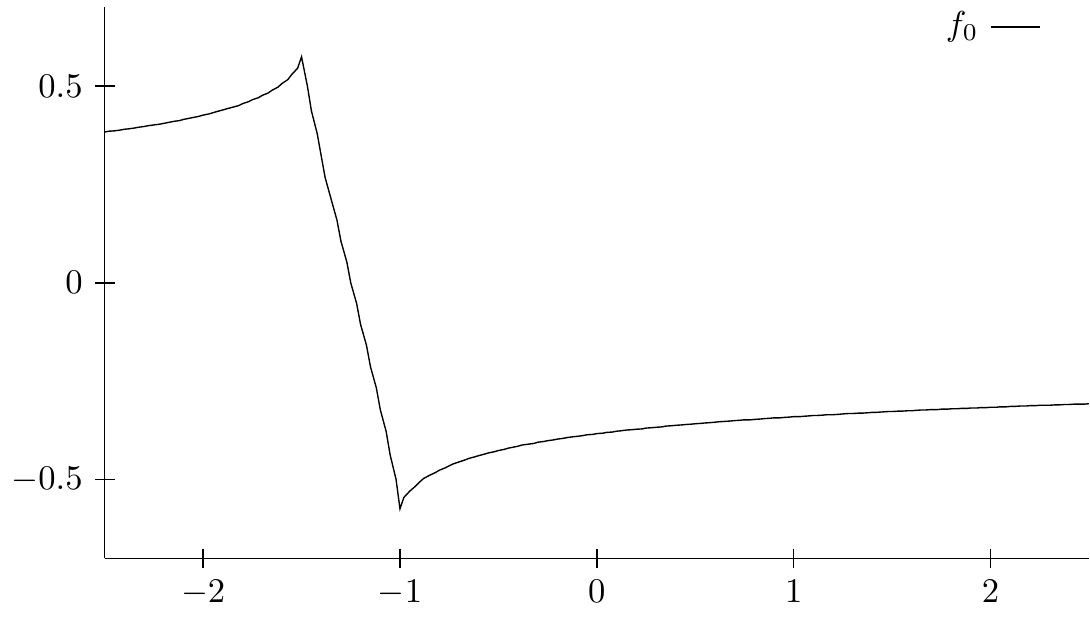}}
	\caption{plot of $f_0 = |x+1|^\alpha - |x+3/2|^\alpha$ for $\alpha = 0.8$}
	\label{fig:f0}
\end{figure*}

\begin{figure*}[!h]
	\makebox[\textwidth][c]{\includegraphics[scale=.8]{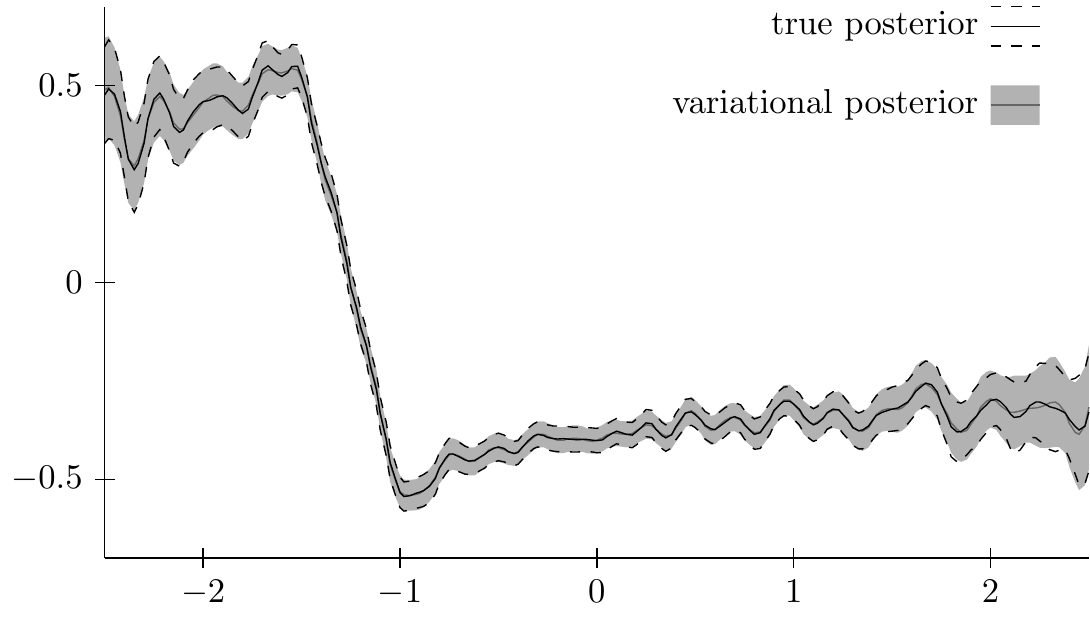}}
	\caption{True and variational posterior and credible regions for squared exponential prior and $m=80$ inducing variables from method 2}
	\label{fig:good}
\end{figure*}

\begin{figure*}[!h]
	\makebox[\textwidth][c]{\includegraphics[scale=.8]{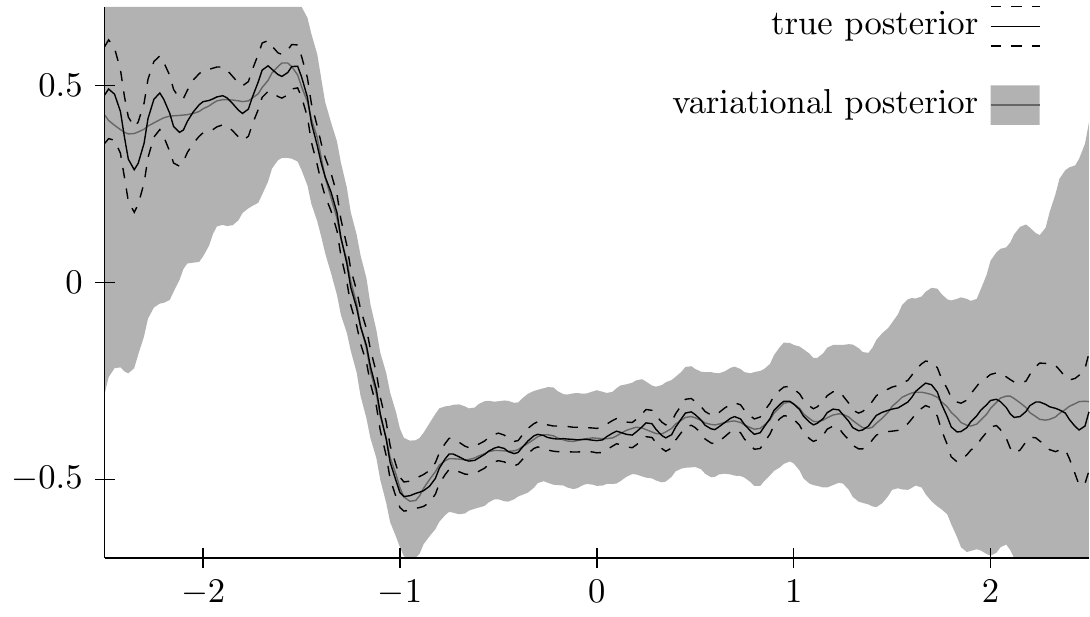}}
	\caption{True and variational posterior and credible regions for squared exponential prior and $m=40$ inducing variables from method 2}
	\label{fig:bad}
\end{figure*}

\newpage \section{Conclusion}

In this paper we consider the inducing variables variational Bayes method for GP regression and determine sufficient conditions, under which the variational approximation achieves the same contract rate around the true functional parameter of interest as the original posterior. As examples we consider three commonly used priors and two choices of inducing variables obtained from spectral decompositions and determine a lower bound on the number of inducing variables, which is sufficient for achieving optimal (minimax) contraction rates for the corresponding variational posterior.

The numerical experiments show that variational credible regions are wider than those associated with the true posterior when too few inducing variables are chosen, providing overly conservative uncertainty statements. Nevertheless this suggests that reliable uncertainty quantification should also carry over from the true to the variational posterior, even if the variational approximations are too sparse. In other words, if the original credible regions can ``capture'' the true regression function with $\P_0$-probability tending to one, then so will variational credible regions. A natural next step is to substantiate these experimental results by theory.

Besides the two choices of inducing variables discussed in this paper, there are various inducing point methods that fall within our framework, simply by taking inducing variables of the form $u_j = f(z_j)$ for points $z_j \in \mathcal X$. \cite{burt2020} discuss several other methods for selecting the inducing points $z_j$ and obtain bounds on the KL-divergence between the true and variational posterior. It would be interesting to see, by means of an application of Theorem \ref{thm: main}, what the minimal number of inducing points has to be in order for these methods to yield optimal contraction rates. 

\acks{We would like to thank the AE and three anonymous reviewers for providing many useful comments that lead to an improved version of the paper.

This project has received funding from the European Research Council (ERC) under the European Union’s Horizon 2020 research and innovation programme (grant agreement No. 101041064). }


\appendix

\section{Theory of contraction rates}\label{a:gpcon}

In this section we provide a brief summary of the frequentist theory of contraction rates for Gaussian Process priors, tailored to our setting. First we start with a general contraction rate result for (nonparametric) posterior distributions. It is a slightly modified version of Theorem 8.9 of \cite{ghosal2017} (which also directly follows from their proof), similar to the original statement that appeared in the seminal paper by \cite{ghosal2000}, but simplified and adapted to our setting. It makes use of the so-called covering number (or entropy)
\begin{equation}\label{e:covnum}
N(\epsilon,\calF,d_{\mathrm H}),
\end{equation}
which is the minimal number of $d_{\mathrm H}$-balls of radius $\epsilon$ required to cover the set $\calF \subset L^2(\calX,G)$.

\begin{lemma}\label{l:ggv}
Suppose that there exists a sieve $\calF\subset L^2(\calX,G)$, a constant $C>0$, and a sequence of postive numbers $\epsilon_n$ with $n\epsilon_n^2 \to \infty$, such that
\begin{gather}
\Pi(f: \|f-f_0\|_{2,G} < \epsilon_n) \geq \exp(-C n\epsilon_n^2), \label{e:sball} \\
\log N(\epsilon_n, \calF, d_{\mathrm H}) \lesssim n \epsilon_n^2, \label{e:entropy} \\
\Pi(\calF^c) \leq \exp(-(C+4)n\epsilon_n^2). \label{e:sprior} 
\end{gather}
Then there exists an event $A_n$ such that $\P_0(A_n) \to 1$, and
\begin{equation*}
\E_0 \Pi(f : d_{\mathrm H}(f,f_0) \geq M_n\epsilon_n \mid \bs x,\bs y)1_{A_n} \lesssim \exp(-C_2 n\epsilon_n^2)
\end{equation*}
holds for some $C_2>0$, and, consequently, the posterior distribution contracts at the rate $\epsilon_n$.
\end{lemma}

The above lemma can be summarised as follows: the posterior contraction rate at $f_0$ is $\epsilon_n$ if the prior puts sufficient mass on $\epsilon_n$-balls around $f_0$, and the parameter space can be divided into two sets, of which one has log-entropy of order $n\epsilon_n^2$, and the other attains exponentially small prior mass.

The original statement of this result differs in two ways from ours. Firstly, the original condition \eqref{e:sball} uses KL-divergence and KL-variation instead of the $L^2$-norm $\|\cdot\|_{2,G}$. In our case the statements are equivalent since these are of the same order (see Lemma 2.7 in \citeb{ghosal2017}). Secondly, the original theorem includes a testing condition, which holds in our case due to the use of the Hellinger distance. For details we refer to Appendix D of \cite{ghosal2017}.

For GP priors the contraction rate can be characterised by the concentration function inequality \eqref{eq: con}, since it replaces the conditions of Lemma \ref{l:ggv}.
Indeed, suppose that \eqref{eq: con} holds for some $\epsilon_n\to 0$ with $n\epsilon_n^2\to \infty$. Theorem 2.1 in \cite{vaart2008} applied to the Banach space $L^2(\calX,G)$ with norm $\| \cdot \|_{2,G}$ yields a sieve $\calF$ such that \eqref{e:sball} and \eqref{e:sprior} hold, and moreover,
\begin{equation}\label{e:cover}
\log N(\epsilon_n,\calF,\|\cdot\|_{2,G}) \lesssim n\epsilon_n^2.
\end{equation}
This means there is a bound for a covering number with respect to a different metric. But the elementary inequality $1-e^{-u} \leq u$ applied to \eqref{e:hel} shows that the Hellinger distance $d_{\mathrm H}$ is bounded by the $L^2$-norm $\|\cdot\|_{2,G}$ up to a multiplicative constant, so the covering number in condition \eqref{e:entropy} is bounded by a constant multiplied by the covering number in \eqref{e:cover}. This means all conditions of Lemma \ref{l:ggv} are satisfied and so Lemma \ref{l:con} is proved.

To connect the contraction rates of the true and variational posterior in the proof of Theorem \ref{thm: main}, we use the following result, which is Theorem 5 of \cite{ray2021}.

\begin{lemma}\label{l:rs}
Let $\calF_n$ be a measurable subset of the parameter space $L^2(\calX,G)$, $A_n$ be an event, and $Q$ a distribution for $f$. If there exist $C,\delta_n > 0$ such that 
\[ \E_0 \Pi(\calF_n \mid \bs x,\bs y) 1_{A_n} \leq C e^{-\delta_n}, \]
then
\[ \E_0 Q(\calF_n) 1_{A_n}\leq \frac{2}{\delta_n}\Big(\E_0 \KL(Q\,\|\,\Pi(\,\cdot\mid\bs x,\bs y)) 1_{A_n} + C e^{-\delta_n/2}\Big). \]
\end{lemma}

Although this theorem was applied in context of a high-dimensional parameter space in  \cite{ray2021}, the result holds for general (possibly infinite-dimensional) parameter spaces, hence can be applied in our setting as well.

\section{The concentration function inequality for the squared exponential prior}\label{a:sqexp}

We provide the lemmas used in the proof of Lemma \ref{lem: conc_function}. The first lemma deals with the $L^2(\calX,G)$-entropy of the unit ball $\bbH^b_1$ of the RKHS of the squared exponential process with length scale $b$. Recall that $N$ is defined in \eqref{e:covnum}.

\begin{lemma}\label{lem: entropy}
Let $f$ be the squared exponential process with covariance function \eqref{eq:sqexp}
and suppose that $G$ satisfies the sub-Gaussian tail bound \eqref{eq: sg}. 
There exist a constant $K > 0$ such that for all small enough $\eps> 0$, the logarithm of the covering number satisfies 
\[
\log N(\eps, \bbH^b_1, \|\cdot\|_{2,G}) \le K b^{-d}\Big(\log\frac1\eps\Big)^{\kappa},
\]
where $\kappa = 1+3d/2$. 
\end{lemma}

\begin{proof}
Let $\mu^b(d\lambda) = (2\pi^{1/2}/b)^{-d}\exp(-\|b\lambda \|^2/4)\,d\lambda$ be the spectral measure of the process $f$. 
By Lemma 4.1 of \cite{vaart2009}
the RKHS of the process is the collection $\bbH^b$ of (real parts of) all functions 
of the form 
\[
h_\psi(x) = \int e^{i\qv{\lambda, x}}\psi(\lambda)\,\mu^b(d\lambda),
\]
with $\psi \in L^2(\mu^b)$, and $\|h_\psi\|_{\bbH^b} = \|\psi\|_{L^2(\mu^b)}$. 
By Cauchy-Schwarz  and the fact that $\mu^b(B) = \mu^1(b B)$ all functions in the RKHS unit ball $\bbH^b_1$ are uniformly bounded by 
$C = \sqrt{\mu^1(\bbR^d)}$. It follows that for $h_1, h_2 \in \bbH^b_1$ and $a > 0$ we have 
\[
\|h_1-h_2\|_{2, G} \le \sup_{\|x\| \le a}|h_1(x) - h_2(x)| + \sqrt{2} C \sqrt{G(x: \|x\|> a)}.
\]
The sub-Gaussianity assumption implies that for $a$ a large enough multiple of  $\sqrt{\log (1/\eps)}$ we have
 $G(x: \|x\|> a) \le \eps^2/(2C^2)$, so that 
\[
\log N\Big(2\eps, \bbH^b_1, L^2(\calX,G)\Big) \le \log N(\eps, \bbH^b_1, \ell^\infty[-a,a]). 
\]
By Lemma 4.5 of \cite{vaart2009} the entropy on the right is bounded by a constant times $(a/b)^d (\log (1/\eps))^{1+d}$.
\end{proof}

Using the well-known connection between the entropy of the RKHS unit ball and 
the small ball probabilities of a centered Gaussian process as in Lemma 4.6 of \cite{vaart2009}, 
we obtain the following small ball estimate from Lemma \ref{lem: entropy}.

\bigskip

\begin{lemma}\label{lem: smallball}
Let $f$ be the squared exponential process with covariance function \eqref{eq:sqexp}
and suppose that $G$ satisfies the sub-Gaussian tail bound \eqref{eq: sg}. 
There exist a constant $K > 0$ such that for all small enough $\eps> 0$ 
\[
-\log \Pi(f: \|f\|_{2, G} \le \eps) \le K b^{-d}\big(-\log (b\eps)\big)^{\kappa},
\]
where $\kappa = 1+3d/2$. 
\end{lemma}

\bigskip

The following lemma deals with the approximation term in the concentration function \eqref{eq: conf}. It follows from the proof of Lemma 4.3 of \cite{vaart2009}. 

\begin{lemma}\label{lem: decentering}
Let $f$ be the squared exponential process with covariance function \eqref{eq:sqexp}
and suppose that $G$ satisfies the sub-Gaussian tail bound \eqref{eq: sg}. Let $\bbH^b$ be 
the RKHS of $f$. If $f_0 \in L^2(\bbR^d) \cap C^\alpha(\bbR^d)$ for $\alpha > 0$, then there exist constants $K_1, K_2 > 0$ 
such that 
\[
\inf_{h\in \bbH^b:\|h-f_0\|_{2, G} \le K_1b^{\alpha}} \|h\|^2_{\bbH^b} \le K_2 b^{-d}
\]
for all $b > 0$ small enough.
\end{lemma}




\bibliography{references-v1}

\end{document}